\documentclass[12pt]{article}
\usepackage{amsmath,amssymb,amsfonts,amsthm,url,xspace,graphicx,color}
\newtheorem{theorem}{Theorem}
\numberwithin{theorem}{section}

\newtheorem{proposition}[theorem]{Proposition}
\newtheorem{corollary}[theorem]{Corollary}

\newtheorem{rmk}{Remark}

\newcommand{\A}{{\mathcal A}}

\newcommand{\Z}{\mathbb{Z}}
\newcommand{\R}{\mathbb{R}}
\newcommand{\C}{\mathbb{C}}

\newcommand{\N}{\mathcal{N}}

\renewcommand{\H}{\mathbb{H}}
\renewcommand{\Re}{\mathrm{Re}}
\renewcommand{\Im}{\mathrm{Im}}

\newcommand{\be}{\begin{equation}}
\newcommand{\ee}{\end{equation}}
\newcommand{\old}[1]{}

\newcommand\blfootnote[1]{%
  \begingroup
  \renewcommand\thefootnote{}\footnote{#1}%
  \addtocounter{footnote}{-1}%
  \endgroup
}

\begin{document}
\title{Gradient variational problems in $\R^2$}
\author{Richard Kenyon\thanks{Department of Mathematics, Yale University, New Haven CT 06520; richard.kenyon at yale.edu. Research supported by NSF DMS-1854272, DMS-1939926 and the Simons Foundation grant 327929.} \and Istv\'an Prause\thanks{
Department of Physics and Mathematics, University of Eastern Finland, P.O. Box 111, 80101 Joensuu, Finland; istvan.prause at uef.fi.}}

\date{}
\maketitle
\begin{abstract}
We prove a new integrability principle for gradient variational problems in $\R^2$,
showing that solutions are explicitly parameterized by $\kappa$-harmonic functions,
that is, functions which are harmonic for the laplacian with varying conductivity $\kappa$,
where $\kappa$ is the square root of the Hessian determinant of the surface tension.
\end{abstract}


\section{Introduction}

We consider a variational principle for a function $h:\Omega \to\R$, $\Omega \subset \R^2$ where the quantity to be minimized,
the {\em surface tension} $\sigma$, depends only on the slope:
\begin{equation}
\label{eq:variationalproblem}
 \min_h\int_\Omega\sigma(\nabla h)\,dx\,dy, \quad h|_{\partial \Omega}=h_0.
\end{equation}
Here $\sigma:\N\to\R$ is assumed smooth and strictly convex in the interior $\mathring \N$; $\N\subset\R^2$ closed and simply connected. We call such a variational problem a \emph{gradient variational problem}. 
The problem involves a \emph{gradient constraint} as $\sigma$ is not defined outside $\N$. It is called \emph{admissible} if there exists a locally Lipschitz continuous extension $h$ of $h_0$ satisfying the gradient constraint $\nabla h \in \N$ a.e.
Strict convexity of $\sigma$ implies
both existence and uniqueness of the solution of an admissible problem, see \cite{DSS}. We will also assume a local ellipticity condition on $\sigma$: its Hessian determinant is nonzero on $\mathring \N$.

Gradient variational problems occur in many different settings. The foremost example is
the case of harmonic functions, for which $\sigma(\nabla h) = |\nabla h|^2$;  harmonic functions 
of course form the cornerstone of complex analysis. Another well-known example is the minimal
surface equation, where $\sigma(\nabla h) = \sqrt{1+|\nabla h|^2}$. Other important examples are 
the dimer model (or stepped surface model), see Figure \ref{bpp}, and its variants, and other
models of statistical mechanics. Up until now these problems have been studied in an {\it ad hoc} way. The present work provides a unified approach which applies to all variational problems of this type.

For a positive real function $\kappa:\R^2\to\R$, called \emph{conductance}, 
a \emph{$\kappa$-harmonic function} is a function $u$ satisfying 
$$\nabla\cdot\kappa\nabla u = 0.$$ This is the \emph{inhomogeneous-conductance Laplace equation},
or $\kappa$-Laplace equation. Our main result, Theorem \ref{main} below, 
is that graphs of solutions to any gradient variational problem (\ref{eq:variationalproblem})
are envelopes of $\kappa$-harmonically moving planes in $\R^3$.
Here $\kappa= \sqrt{\det\text{Hess}(\sigma)}$. 

We identify a large class of surface tensions for which this allows us to give explicit solutions.
We say that $\sigma$ \emph{has trivial potential} if its Hessian determinant is the fourth power of a harmonic function of the intrinsic coordinate (see definition below). We show that in this case one can solve (\ref{eq:variationalproblem}) in a very concrete sense,
called \emph{Darboux integrability}:
one can find explicit parameterizations of all solutions in terms of analytic functions. 

We discuss several representative examples, including 
the dimer model of \cite{KO1}, 
the ``enharmonic laplacian" $\sigma(s,t) = -\log st$ of \cite{AK}, and the $p$-laplacian $\sigma(s,t)=(s^2+t^2)^{p/2}$ 
and other isotropic surface tensions. An important example with trivial potential, the $5$-vertex model \cite{dGKW},
is discussed and worked out in detail in \cite{KP2}. 

\paragraph{Probabilistic applications.} While our results apply to any gradient variational problem, 
this work was originally motivated by probabilistic applications. ``Limit shapes'' for several probability models are known (or conjectured) to be minimizers of gradient variational problems. Let us briefly recall the limit shape problem, see Figure \ref{bpp} for an example, and \cite{CKP}.
\begin{figure}
\center{\includegraphics[width=4in]{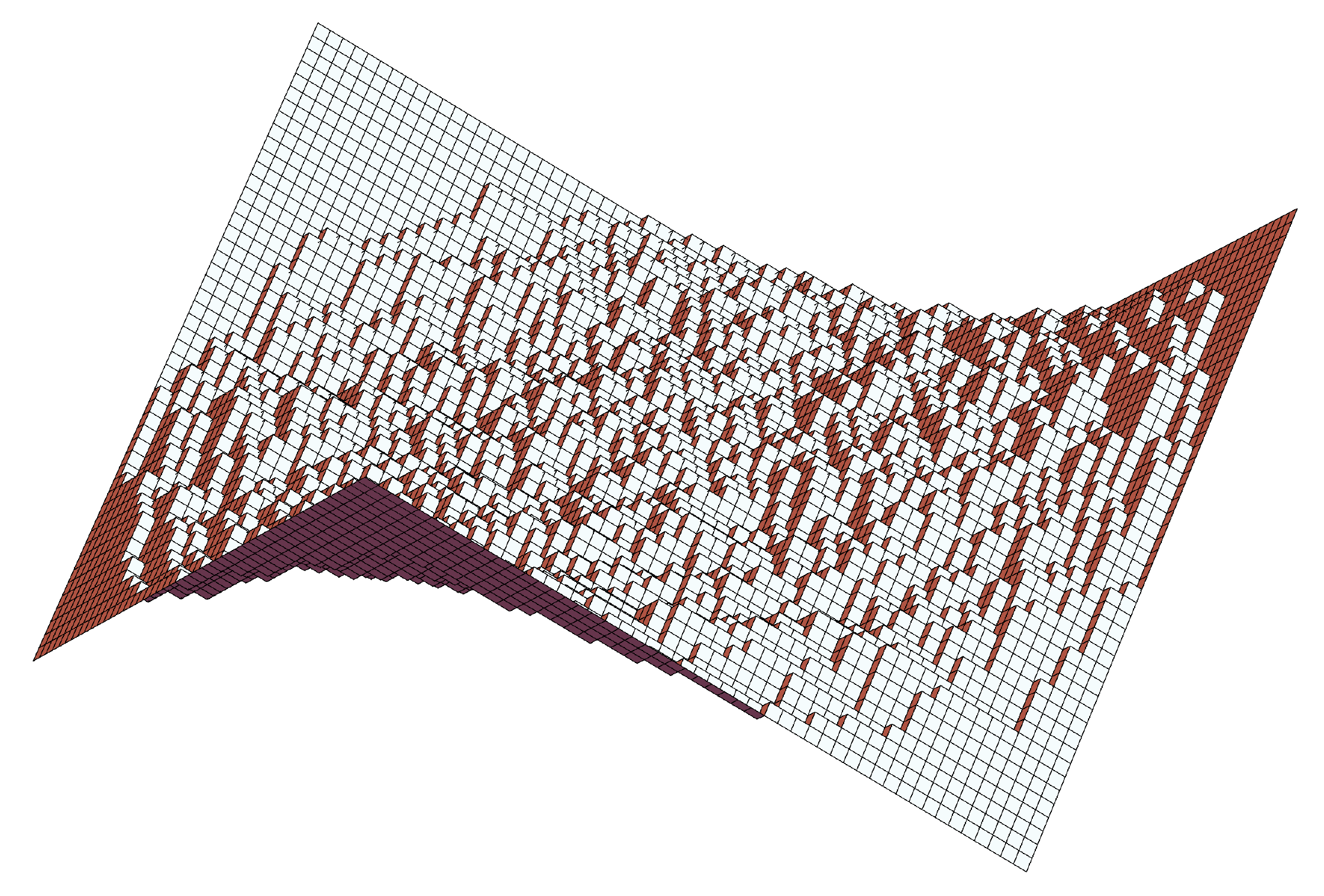}}
\caption{\label{bpp}A uniform random sample of a large ``Boxed Plane Partition" (shown here for a box of size $n=50$).
This is a \emph{stepped surface} spanning six edges of an $n\times n\times n$ cube. For $n$ large a random sample
lies (with probability tending to $1$ as $n\to\infty$) near a certain fixed smooth surface, which was first computed by Cohn, Larsen and Propp in \cite{CLP}.}
\end{figure}
Many well-known 
statistical mechanical models such as the plane partition model, the dimer model and the five- or six-vertex models, 
are models of random discrete Lipschitz functions from $\Z^2$ to $\Z$. The \emph{limit shape problem} 
is to understand the shape of the random function with fixed Dirichlet
boundary conditions in the \emph{scaling limit}, that is, in the limit when the lattice spacing tends to zero. In quite general
situations the random surface concentrates, in the scaling limit, onto a nonrandom continuous surface,
which is obtained by minimizing a gradient variational problem, where the surface tension
function encodes the ``local entropy'' of the probability model. 

For many determinantal models, the exact surface tension is known, see \cite{CKP,KOS} for the dimer model and \cite{WSun} for random Young tableaux. These have the property that $\kappa$ is constant and hence $\kappa$-harmonic functions are simply harmonic. In this situation we obtain a surprisingly simple representation of limit shapes: they are envelopes of harmonically moving planes in $\R^3$, see Corollary \ref{cor:dimerharmonic}. The advantage of this representation compared to those based on complex Burgers equation \cite{KO1,ADPZ} is that it makes the problem of matching a limit shape to given boundary values much more feasible and systematic, since often there is no need to guess the frozen boundary \cite{KP3}. We illustrate this by two examples in Section \ref{se:limitshapes}.

An essential novelty of our work lies in the fact that it also applies to \emph{non-determinantal} models.
Our prime example of a surface tension with trivial potential is in fact the surface tension arising in the five-vertex model. As far as we know, the five-vertex model \cite{dGKW} and its genus-zero generalisation \cite{KP2} are the only examples beyond determinantal models where the exact surface tension has been derived. With the help of the trivial potential property all limit shapes of these models can be explicitly parametrised \cite{KP2}.

In the current paper we borrow some terminology from the probability setting such as free energy, liquid region and amoeba;
see their definitions for general gradient models below. 

As a final note, it would be interesting to compare the ``tangent plane method'' of the present work and \cite{KP3} to the tangent (line) method of \cite{CS} introduced in the context of the six-vertex model.
\medskip

\noindent{\bf Acknowledgements.} We thank Robert Bryant for discussions and 
hints on Amp\`ere's theorem and Darboux integrability. 
We thank Filippo Colomo and Andrea Sportiello for discussions on arctic curves and the tangent method. We are grateful to the referees for their comments, suggestions and careful reading of the manuscript.

\section{Intrinsic coordinate}

Let $(s,t)$ be coordinates for $\N\subset\R^2$. 

\paragraph{Free energy.}
The Legendre dual $F(X,Y)$ to the surface tension $\sigma$ is defined for every $(X,Y) \in \R^2$ as
\[ F(X,Y)=\sup_{(s,t) \in \N} \left( sX+tY-\sigma(s,t) \right).
\]
In probabilistic settings $F$ is called the \emph{free energy}.
We define the {\em amoeba}\footnote{It is generally not an algebraic amoeba; the terminology is motivated by analogy with the dimer model.} $\A=\A(F) \subset \R^2$ of $F$ to be the closure of the set where $F$ is strictly convex. 
The gradient map $\nabla \sigma \colon \N \to \A$ gives a one-to-one correspondence between the interiors
$\mathring\N$ and $\mathring\A$.
By the Wulff construction \cite{DKS}, $F$ is a volume-constrained minimizer for $\sigma$, that is, a minimizer of the
surface tension functional (\ref{eq:variationalproblem}) under the additional constraint of having
fixed volume under the graph
of $h$, and for appropriate boundary conditions at $\infty$.

\paragraph{Isothermal coordinates and complex structure.}

The Hessian matrix of $\sigma$, $H_\sigma$, is positive definite and so determines a Riemannian metric on 
$\mathring\N$ 
$$g=\sigma_{ss}\,ds^2 + 2\sigma_{st}\,ds\,dt + \sigma_{tt}\,dt^2.$$
There is likewise a metric on $\mathring \A$ determined by the Hessian of $F$; the map $\nabla \sigma: \mathring \N\to \mathring \A$ 
is however an isometry for these metrics. 

Let $z=u+iv=u(s,t)+iv(s,t)$ be a conformal (also known as isothermal) coordinate system for $g$, that is, 
$g =e^{\Phi}(du^2+dv^2)$
for some function $\Phi=\Phi(u,v)$. To match with the usual definitions for the dimer model, we will choose
$z$ to be an orientation \emph{reversing} homeomorphism to some uniformizing domain, which we can choose to be 
either $\H$, the upper 
half-plane\footnote{In the dimer model $\sigma$ is smooth
on $\mathring\N$ except at a finite number of points, where it has conical singularities; these singularities
lead to holes in $\A$.  It is then natural to parameterize $\A$ with a multiply-connected domain
\cite{KOS}. 
For our purposes here, however, $\A$ is assumed simply connected.}
or $\C$. 
We call $z$ the \emph{intrinsic coordinate}.

To find $z$, following Gauss (see e.g. \cite{Spivak4}), we solve the Beltrami equation
\be\label{Belt0} \frac{\bar \partial z}{\partial z}=\frac{\frac12(z_s+i z_t)}{\frac12(z_s-i z_t)}=\frac{1}{\bar \mu_\sigma},
\ee
where $\mu_\sigma$ is the \emph{Beltrami coefficient}
\[ \mu_\sigma=\frac{\sigma_{ss}-\sigma_{tt}+2i \sigma_{st}}{\sigma_{ss}+\sigma_{tt}+2 \sqrt{\sigma_{ss}\sigma_{tt}-\sigma_{st}^2}}.
\]
Equivalently, define  the {\em Gauss map} (or ``complex slope'') $\gamma$ on $\N$ to be 
\begin{equation} 
\label{eq:gamma}
\gamma:=\frac{-\sigma_{st}-i\sqrt{\sigma_{ss}\sigma_{tt}-\sigma_{st}^2}}{\sigma_{ss}}=\frac{\sigma_{tt}}{-\sigma_{st}+i \sqrt{\sigma_{ss}\sigma_{tt}-\sigma_{st}^2}},
\end{equation}
then $z$ satisfies the following equation
\be\label{Belt1} \frac{z_s}{z_t}=-\frac1{\bar \gamma}.
\ee
The local ellipticity of $\sigma$ implies that $|\mu_\sigma|<1$ and uniformly bounded away from $1$ on compact subsets of $\mathring \N$ which guarantees the existence of the solution for the Beltrami equation in \eqref{Belt0} \cite{AIM}. 

We can  equivalently work on $\A$ and define $z$ by the equation
\be\label{Belt2}\frac{z_X}{z_Y}= \bar \gamma.\ee 

Since $\bar z_XX_z + \bar z_YY_z=0$, and $\bar z_ss_z + \bar z_tt_z=0$, equations (\ref{Belt1}) and (\ref{Belt2}) lead to 
\be\label{othergamma} \gamma=-\frac{Y_z}{X_z}=\frac{s_z}{t_z}. 
\ee

With this setup $X,Y$ and $s,t$ are all functions of the intrinsic coordinate $z$, and satisfy (\ref{othergamma}).

\begin{proposition}
\label{prop:isothermal}
The following statements are each equivalent to $\zeta$ being an intrinsic coordinate
\begin{enumerate}
\item[(i)]
$ \frac{X_\zeta}{t_\zeta}+\frac{Y_\zeta}{s_\zeta}=0$,
\item[(ii)]
$\frac{s_\zeta}{t_\zeta}=\frac{-\sigma_{st} - i\sqrt{\det H_\sigma}}{\sigma_{ss}},
$
\end{enumerate}
and when either of these holds we necessarily have \be\label{iii}
\frac{X_\zeta}{t_\zeta}=- i \sqrt{\det H_\sigma}=-\frac{Y_\zeta}{s_\zeta},
\ee
\end{proposition}

\begin{proof}
Define $\gamma$ by $\gamma=\frac{s_\zeta}{t_\zeta}$. Since $X_\zeta=(\sigma_s)_\zeta=\sigma_{ss} s_\zeta + \sigma_{st} t_\zeta$ we can write 
$\frac{X_\zeta}{t_\zeta}=\sigma_{ss} \gamma +\sigma_{st}$.
Similarly, 
$\frac{Y_\zeta}{s_\zeta}=\sigma_{tt} \frac{1}{\gamma} +\sigma_{st}$. Now $(i)$ is equivalent to the equation
\[ \sigma_{ss} \gamma^2 +2\sigma_{st} \gamma +\sigma_{tt}=0,
\]
which has two solutions
\[ \gamma=\frac{-\sigma_{st} \pm i \sqrt{\det H_\sigma}}{\sigma_{ss}}.
\]
Because $\zeta$ is assumed to be orientation reversing, $\Im \gamma<0$ and thus we have the minus sign above. This leads to $\frac{X_\zeta}{t_\zeta}=-i \sqrt{\det H_\sigma}=-\frac{Y_\zeta}{s_\zeta}$. Thus $(i),(ii)$ are equivalent
(and imply $(\ref{iii})$). On the other hand by definition the intrinsic complex variable $\zeta$ is an orientation-reversing homeomorphism $\zeta \colon \N \to \H$ solving (\ref{Belt1}); this is equivalent to $(ii)$ in terms of the inverse mapping.
\end{proof}

\old{
\begin{proposition}
\label{prop:isothermal}
The complex derivatives satisfy the following relation
\[ \frac{X_z}{t_z}=-i \sqrt{\det H_\sigma}=-\frac{Y_z}{s_z}.
\]
Conversely, the intrinsic coordinate $z$ is characterized by these equations.
\end{proposition}

\begin{proof}
Since $X_z=(\sigma_s)_z=\sigma_{ss} s_z+\sigma_{st} t_z$, dividing by $t_z$ and using (\ref{othergamma}) and 
(\ref{eq:gamma})  we have
$X_z/t_z=-i \sqrt{\det H_\sigma}$.
Then (\ref{othergamma}) gives the other equality.
\end{proof}
}

\paragraph{Intrinsic Euler-Lagrange equation.}
Consider a minimizer $h$ of the variational problem \eqref{eq:variationalproblem}. A priori regularity results for the minimizer are established in \cite{DSS}.
The {\em liquid region} $\mathcal{L} \subset \Omega$ is defined to be the subset where $\nabla h$ is in the interior of $\N$
and thus $\sigma$ is smooth. On the liquid region $h$ satisfies the associated Euler-Lagrange equation 
\begin{equation}
\label{eq:EL} 
\mbox{div}(\nabla \sigma \circ \nabla h)=0,
\end{equation}
or simply $X_x+Y_y=0$.
Amp\`ere showed that, combined with the fact that $\nabla h$ is curl-free we get a single complex equation for $z \colon \mathcal{L} \to \C$:
\begin{theorem}[Amp\`ere \cite{Amp}]
\label{thm:intrinsicEL}
We have
\be
\label{Ampere}X_zz_x + Y_z z_y=0 \quad (x,y) \in \mathcal{L}.
\ee
In an equivalent form,
\[ \frac{z_x}{z_y}=\frac{s_z}{t_z}=\gamma.
\]
\end{theorem}

\begin{proof}
Observe that
\begin{align*} 
X_x-i \sqrt{\det H_\sigma} t_x=X_z z_x + X_{\bar z} \bar z_x -i \sqrt{\det H_\sigma} \left( t_z z_x + t_{\bar z} \bar z_x \right)\\
= \left( X_z - i \sqrt{\det H_\sigma} t_z \right) z_x +\overline{\left( X_z + i \sqrt{\det H_\sigma} t_z \right)z_x}= 2 X_z z_x 
\end{align*}
in view of Proposition \ref{prop:isothermal}. Similarly,
\[ Y_y+i \sqrt{\det H_\sigma} s_y=2 Y_z z_y. 
\]

Thus the Euler-Lagrange equation $X_x+Y_y=0$ together with the curl-free condition $t_x=s_y$ combine to a single complex equation
\[ X_z z_x+ Y_z z_y=0.
\]
\end{proof}

\begin{rmk} We can make a volume-constrained problem by fixing the volume under the graph of $h$.  
The volume constrained Euler-Lagrange equation is $2X_z z_x+2 Y_z z_y=c$, where $c\in\R$ is the Lagrange multiplier for the volume, see \cite{KO1}.
\end{rmk}

\paragraph{Complex structure on the liquid region.}
There is a canonical mapping from the liquid part $\mathcal{L}$
of the minimizer to the Wulff shape, defined in terms of the corresponding tangent planes. In planar coordinates, $(x,y) \in \mathcal{L} \mapsto (\nabla \sigma \circ \nabla h)(x,y) \in \mathcal{A}$. Composed with the isothermal coordinate, $z \colon \mathcal{L} \to \C$ gives an intrinsic complex structure to $\mathcal{L}$ with 
\begin{equation}
\label{eq:cpxliquid}
\frac{z_x}{z_y}=\gamma.
\end{equation}
Thus in the intrinsic complex structure the map $\mathcal{L} \to \mathcal{A}$ is holomorphic. For surface tensions arising from the dimer model, equation \eqref{eq:cpxliquid} is equivalent to the complex Burgers equations of \cite{KO1}, or to the Beltrami equations studied in \cite{ADPZ}.

\section{$\kappa$-harmonic functions}
We call a (sufficiently smooth) solution $w$ to the conductivity equation in a domain $D \subset \R^2$
\begin{equation}
\label{eq:conductivity}
 \nabla \cdot \kappa \nabla w =0
\end{equation} 
a {\em $\kappa$-harmonic function}\footnote{the conventional terminology is $\sigma$-harmonic but we reserve $\sigma$ for the surface tension.  See e.g. Chapter 16 of \cite{AIM} for more on $\kappa$-harmonic functions.}. The conductivity $\kappa \colon D \to (0,\infty)$ in our setting is smooth, indeed we set $\kappa(z)=\sqrt{\det H_\sigma}$. That is, we consider the Hessian determinant of $\sigma$ (in $(s,t)$-coordinates) as a function of the intrinsic coordinate $z \in D$. Here $D$ is $\H$ or $\C$. For later purposes, we also introduce the {\em fourth root}
\be 
\label{eq:intrinsichessian}
\psi(z):= \sqrt[4]{\det H_\sigma}=\kappa(z)^{1/2}>0.
\ee
We interpret Proposition \ref{prop:isothermal} in real notation as follows. Recall that $z=u+iv$ and the 
Hodge star operator acts as a (counterclockwise) rotation by $90$ degrees, i.e. $\ast \nabla=\ast (\partial_u,\partial_v)=(-\partial_v,\partial_u)$,
\be\label{nabXY}\nabla X=\ast \kappa \nabla t \quad \text{and} \quad \nabla(-Y)=\ast \kappa \nabla s.
\ee
Since the Hodge star operator transforms curl-free fields into divergence-free fields, it follows that
\begin{equation}
\label{eq:k-harmonicst}
\nabla \cdot \kappa \nabla t=0 \quad \text{and} \quad \nabla \cdot \kappa \nabla s=0,
\end{equation}
and thus $t$ and $s$ are both $\kappa$-harmonic functions in $D$. Equation (\ref{nabXY}) shows that 
by definition $X$ and $-Y$ are the respective conjugate functions. These are in turn $1/\kappa$-harmonic functions.

\paragraph{Reduction to Schr\"odinger equation.}
A standard technique is to reduce \eqref{eq:conductivity} to the Schr\"odinger equation
\begin{equation}
\label{eq:Schrodinger}
(-\Delta+q)(\tilde w)=0,
\end{equation}
with potential $q=\frac{\Delta \psi}{\psi}$. Indeed, with the substitution $\tilde w=w \psi$ satisfies \eqref{eq:Schrodinger}. 
Define real-valued functions $\phi,\phi^*$ as $\phi=s\psi$ and $\phi^*=t\psi$. From \eqref{eq:k-harmonicst} and the above reduction
\be
\label{eq:phi}
\psi \Delta\phi - \phi \Delta\psi =0 \quad \text{and} \quad \psi \Delta\phi^* - \phi^* \Delta\psi =0,
\ee
which can also be verified directly from Proposition \ref{prop:isothermal}.

\paragraph{The intercept function.}
Consider next the \emph{intercept function} $h-(sx+ty)$ in the liquid region $(x,y) \in \mathcal{L}$, where $h$ is the minimizer, and $(s,t)=\nabla h(x,y)$. We will also view this function as a function of the intrinsic coordinate $z$; now however this function is multi-valued, as $(x,y) \mapsto z$ is generally many-to-one. Alternatively, we can consider the intercept function as a single-valued function in the liquid region in its intrinsic complex structure.
The next theorem shows that the intercept function is $\kappa$-harmonic with respect to this intrinsic variable.

\begin{theorem}
\label{thm:k-harmonic}
The functions $s$, $t$ and $h-(sx+ty)$ are all $\kappa$-harmonic in the liquid region with the respect to the intrinsic coordinate $z$.
\end{theorem}

\begin{proof}
The statement for $s$ and $t$ is just a repetition of \eqref{eq:k-harmonicst}. We also record from
\eqref{eq:phi} 
\be
\label{eq:sandt}
 (s \psi)_{z \bar z}=\phi_{z \bar z}=s \psi_{z \bar z}=0 \quad \text{and} \quad (t \psi)_{z \bar z}=\phi^*_{z \bar z}=t \psi_{z \bar z}=0.
\ee
Consider next
\[ (h-(sx+ty)) \psi. 
\]
Let us take the $\partial_z$ derivative, and remembering that $\nabla h=(h_x,h_y)=(s,t)$
\begin{align*}
(\psi (h-(sx+ty)))_z&=\psi_z (h-(sx+ty))- \psi (s_z x + t_z y)\\
&=\psi_z (h-(sx+ty))+(s \psi_z-\phi_z)x+(t \psi_z - \phi^*_z)y.
\end{align*}
Take now the $\partial_{\bar z}$ derivative and observe a number of cancellations 
\begin{align*}
(\psi (h-(sx+ty)))_{z \bar z}=&
\psi_{z \bar z}(h-(sx+ty))-\psi_z(s_{\bar z}x+t_{\bar z} y)+s_{\bar z} \psi_z x + t_{\bar z} \psi_z y\\
&+(s \psi_{z \bar z}-\phi_{z \bar z}) x+(t \psi_{z \bar z}-\phi^*_{z \bar z}) y+(s \psi_z-\phi_z) x_{\bar z} + (t \psi_z- \phi^*_z) y_{\bar z} \\
=&\psi_{z \bar z} (h-(sx+ty))-\psi(s_ z x_{\bar z}+ t_z y_{\bar z}).
\end{align*}
We now apply Theorem \ref{thm:intrinsicEL} (the intrinsic Euler-Lagrange equation) for the inverse relation -- away from critical points -- $z \mapsto (x,y)$ in the form $-y_{\bar z}/x_{\bar z}=z_x/z_y=s_z/t_z$ to find altogether
\be
\label{eq:h-sx-ty}
 ((h-(sx+ty)) \psi)_{z \bar z}= (h-(sx+ty)) \psi_{z \bar z}.
\ee
This means that $(h-(sx+ty)) \psi$ is a solution to \eqref{eq:Schrodinger}, in other words, $h-(sx+ty)$ is a solution to \eqref{eq:conductivity}, i.e. it is $\kappa$-harmonic.
\end{proof}

\paragraph{Envelopes.} 
The graph of the minimizer $h$ is a surface in $\R^3$. Over a point $(x,y) \in \mathcal{L}$ there is a tangent plane to the minimizer 
which has slope $(s,t)=\nabla h (x,y)$ and it intersects the vertical axis at the point $h(x,y)-(sx+ty)$. 
That is, using $(x,y,x_3) \in \R^3$ coordinates, the tangent plane is $P_z =\{x_3 = sx+ty+c\}$ where $c=h-(sx+ty)$. Theorem \ref{thm:k-harmonic} says that all the coefficients $s,t,c$ of these tangent planes are $\kappa$-harmonic with respect to $z$. In other words, we have shown:

\begin{theorem}\label{main} The minimizer is an envelope of $\kappa$-harmonically moving planes in $\R^3$.
\end{theorem}

In circumstances where we can determine a priori the values of $z$ along the boundary of $\Omega$,
and where we can solve the $\kappa$-laplace equation with these boundary values,
this allows us to solve the problem (\ref{eq:variationalproblem}).

\section{Trivial potential}
\label{se:trivpotential}

We say that the surface tension $\sigma$ {\em has trivial potential} if $\psi=\kappa^{1/2}$ in \eqref{eq:intrinsichessian} is a harmonic function of $z$. In this case, the potential $q \equiv 0$ in \eqref{eq:Schrodinger} and thus $\kappa$-harmonic functions are \emph{ratios} of harmonic functions with a common denominator $\psi$. This leads to the following theorem.

\begin{theorem}
\label{thm:envelope}
If $\sigma$ has trivial potential then $s$, $t$ and $h-(sx+ty)$ are all ratios of harmonic functions (in $z$) in the liquid region with the common denominator $\psi(z)$.
\end{theorem}

For the dimer model, the Hessian determinant is {\em constant}, hence the surface tension has trivial potential.

\begin{corollary}
\label{cor:dimerharmonic}
Minimizers in the dimer model are envelopes of harmonically moving planes.
\end{corollary}

\begin{proof}
By Theorem 5.5 of \cite{KOS}, $\det H_\sigma \equiv \pi^2$ for the dimer model. It means that for the normalized surface tension $\frac{1}{\pi} \sigma$, $\kappa \equiv 1$ and we have harmonic dependence for each coefficient.
\end{proof}

\begin{rmk}
For statistical mechanical models, $\pi/\kappa$ is known in the physics literature as the \emph{Luttinger parameter}, see e.g. \cite{BD,JMS}. For free fermionic systems, like the dimer model, the Luttinger parameter is the constant $1$.
\end{rmk}

The five-vertex model \cite{dGKW} is not free-fermionic and the Hessian determinant is \emph{not} constant. Nevertheless, its surface tension, as well that of its generalization, the genus-zero five-vertex model,
has trivial potential, see \cite{KP2}.

Theorem \ref{thm:envelope} leads to an algorithm for matching a minimizer for (certain) extremal boundary values, namely, those for which the boundary tangent planes are determined. 
This is discussed in \cite{KP3}; see Section \ref{se:limitshapes} below for an illustration of the method.

\paragraph{Solving the Euler-Lagrange equation.}

According to Theorem \ref{thm:envelope} if $\sigma$ has trivial potential then
\begin{equation}
\label{eq:G}
G(z)=\psi(h-(sx+ty))
\end{equation}
is a (possibly multi-valued) harmonic function. We now describe how to parametrize solutions in terms of this harmonic function. We have
$$h_z = sx_z+ty_z = (sx+ty)_z-s_zx-t_zy$$ from which we get
\[ s_z x +t_z y + (h-(sx+ty))_z =0,
\]
or
\begin{equation}
\label{eq:complexline}
s_z x + t_z y + (G/\psi)_z=0.
\end{equation}
The functions $s,t$ and $\psi$ are functions of $z$ depending only on the surface tension $\sigma$ while $G$ depends on the boundary conditions of $h$. The equation \eqref{eq:complexline} describes a complex line through $(x,y) \in \R^2$ with three complex coefficients. Since the complex slope $\frac{s_z}{t_z}=\gamma$ is not real (as $\Im \gamma <0$), this defines uniquely $(x,y)$ in terms of $z$,
\be\label{xandy} x(z)=-\frac{\Im(t_z^{-1} (\frac{G}{\psi})_z)}{\Im(\gamma)}, \quad y(z)=-\frac{\Im(s_z^{-1} (\frac{G}{\psi})_z)}{\Im(\gamma^{-1})}.
\ee
Finally, the solution $h(x,y)=sx+ty+G/\psi$ is also a function of $z$
\be\label{h} h(z)=\frac{G}{\psi} - \frac{\Im \left( \frac{s}{t_z}(G/\psi)_z \right)}{\Im \gamma} 
-\frac{\Im \left(\frac{t}{s_z}(G/\psi)_z \right)}{\Im (\gamma^{-1})}
\ee

This gives an explicit solution to the Euler-Lagrange equation for an \emph{arbitrary} harmonic function $G(z)$.
Typically, however, matching $G(z)$ with the desired Dirichlet boundary conditions on $h$ is a nontrivial problem. It is also worth noting that for a general $G$, (\ref{xandy}) and (\ref{h}) will solve the Euler-Lagrange equation \emph{locally};
it may
globally correspond to a self-intersecting surface and/or have ramification points.

\section{Examples of gradient variational problems}

\paragraph{Trivial potential example.}

For a first example, let $$\sigma(s,t) =\frac23(s^{-2} + t^{-2})$$
for $s,t\in(0,\infty)$. 
Then 
$H_\sigma=\begin{pmatrix} 4s^{-4}&0\\0&4t^{-4}\end{pmatrix}$,
and $z=s^{-1} - it^{-1}$ is a conformal coordinate.
We have $\det H_\sigma = 16(st)^{-4}$ so $\psi(z) = -\Im(z^2)$ is harmonic.
Also $s_z = \frac{-2}{(z+\bar z)^2}$ and $t_z = \frac{2i}{(z-\bar z)^2}.$
In the simplest case when $G$ is constant, $G\equiv C$, we get
$$(x,y,h) = (\frac{C}{2\Im z}, -\frac{C}{2\Re z},\frac{C}{\Im(z^2)}),$$
or equivalently
$$h(x,y) = (-2/C)xy.$$

\paragraph{Young tableaux.}
The surface tension for random Young tableaux was derived in \cite{WSun}:
$$\sigma(s,t) = -(1+\log\left(\frac{\cos{\pi s}}{\pi t}\right))t$$
where 
$$(s,t)\in(-\frac12,\frac12)\times(0,\infty).$$
This is obtained from a certain limit of the dimer model.

We have 
$$H_\sigma = \begin{pmatrix}\pi^2t\sec^2(\pi s)&\pi\tan(\pi s)\\\pi\tan(\pi s)&\frac1t\end{pmatrix}$$
with $\det H_\sigma \equiv \pi^2$, thus $\sigma$ has trivial potential. A conformal coordinate is 
$$z=-\pi t(\tan(\pi s)-i).$$

Now $\gamma=-\overline{\frac{z_t}{z_s}}=\frac{1}{z}$ and
equation (\ref{Ampere}) becomes
$$z_x-\frac{z_y}{z}=0$$
which is the complex Burgers equation. Using $z_xx_{\bar z}+z_yy_{\bar z}=0,$ this becomes
$y_{\bar z}+\frac{1}{z} x_{\bar z}=0$, and integrating
$$y+ \frac{x}{z}+f(z)=0$$
for an arbitrary analytic function $f$. 
Alternatively, solutions are envelopes of harmonically moving planes.

\paragraph{Enharmonic functions.}
In this case $\sigma(s,t) = -\log(st)$ where $s,t\in[0,\infty)$, see \cite{AK}. 
We have 
$$H_{\sigma} = \begin{pmatrix}\frac1{s^2}&0\\0&\frac1{t^2}\end{pmatrix}.$$
A conformal coordinate is $z=u+iv = \log s-i\log t$. Then $\psi = \frac1{\sqrt{st}}= e^{(-u+v)/2}$
and the potential is $q=\tfrac{\Delta\psi}{\psi} = \frac12,$ a constant.

Amp\`ere's equation (\ref{Ampere}) becomes
\be\label{uxyv}x_{\bar z}+ie^{-u-v}y_{\bar z}=0.\ee
Equivalently, 
$$\nabla y = *e^{u+v}\nabla x,$$
that is, $x$ is $e^{u+v}$-harmonic with conjugate $y$. 
This leads to $$\Delta x+x_u+x_v=0.$$
Solutions to this are linear combinations of $x=e^{au+bv}$ where $a^2+b^2+a+b=0$.
For this elementary solution
the corresponding conjugate $y$ is $y=-\frac{b}{a+1}e^{(a+1)u+(b+1)v}$. 
We then have
$$h_{\bar z} = sx_{\bar z} + ty_{\bar z}=e^ux_{\bar z} + e^{-v}y_{\bar z}=(1+i)(a+ib)e^{(a+1)u+bv}$$
whence
$$h=\frac{(a-b)}{a+1}e^{(a+1)u+bv}.$$
Thus a general solution to the variational problem can parameterized as a linear combination of
$$(x,y,h) = (e^{au+bv},-\frac{b}{a+1}e^{(a+1)u+(b+1)v},\frac{(a-b)}{a+1}e^{(a+1)u+bv})$$
where $(a,b)$ run over solutions to $a^2+b^2+a+b=0$. 

\old{
\paragraph{The $5$-vertex model}
For the $5$-vertex model with parameter $r<1$, \see \cite{}, 
$\N$ is the triangular region bounded by the positive axes and the hyperbola $1-s-t-(1/r^2-1)st=0$. 
The surface tension has a complicated formula but the intrinsic coordinate is
$z$ where 
$$\det H_\sigma = \frac1{\pi^2}(\arg\frac{z}{1-z})^4$$
so that $\psi=\pi^{-1/2}\arg\frac{z}{1-z}$ and the model has trivial potential.
We have $s=\frac{\arg z}{\sqrt{\pi}\psi}$ and $t=\frac{\sqrt{\pi}\arg w}{\psi}$ where $1-w-z+(1-r^2)zw=0$. 
}

\paragraph{The $p$-laplacian.}
Here $\sigma(s,t) = (s^2+t^2)^{p/2} = r^p.$
A conformal coordinate is $z=r^\alpha e^{-i\theta}$ where $re^{i\theta} = s+it$ and $\alpha = \sqrt{p-1}$.
Also $\det H_\sigma = p^2(p-1)|z|^{(2p-4)/\alpha}$ and 
$$q=\frac{\Delta\psi}{\psi} = \frac{(p-2)^2}{4(p-1)|z|^2} = \frac{C}{|z|^2}.$$
To solve 
$(\Delta - q)\tilde w = 0,$ use separation of variables: write $z=Re^{-i\theta}$ and look for solutions of the form
$\tilde w(z) = A(R)B(\theta).$ The equation is
$$(\frac{\partial^2}{\partial R^2} + \frac1R\frac{\partial}{\partial R}+\frac1{R^2}\frac{\partial^2}{\partial \theta^2} - \frac{C}{R^2})A(R)B(\theta) = 0,$$
so $$\frac{A''(R)+\frac1RA'(R)-\frac{C}{R^2}A(R)}{A(R)/R^2} = -\frac{B''(\theta)}{B(\theta)}.$$
Setting both sides to a constant $c$ gives the solution
$$B(\theta) = C_1e^{i\sqrt{c}\theta} + C_2e^{-i\sqrt{c}\theta}$$
$$A(R) = c_1R^{\sqrt{C-c}} + c_2 R^{-\sqrt{C-c}}.$$

Other isotropic surface tensions (surface tensions depending only on $r=\sqrt{s^2+t^2}$) 
can be dealt with in the same way:
a conformal coordinate can be chosen of the form $z=f(r)e^{i\theta}$, and $\psi$ is a function of $r$ only. 
The Euler-Lagrange equation can then be solved by separation of variables.

\section{Limit shape examples}
\label{se:limitshapes}

We demonstrate applications of Section \ref{se:trivpotential} to limit shapes in probability. We consider two concrete illustrative boundary value problems for domino tilings. More complex examples, with a more systematic treatment, are discussed in \cite{KP3} and \cite{P}. See also Section 4.3--4.4 of \cite{KP2} for explicit examples of limit shapes for staggered 5-vertex models. These non-determinantal examples go well beyond the framework of \cite{KO1,ADPZ} and are essentially based on Theorem \ref{thm:envelope} above.

\subsection{Domino tilings of the Aztec diamond}
We start by revisiting the classic arctic circle limit shape for domino tilings of the Aztec diamond, first derived in \cite{CEP}.
For domino tilings (tilings of regions in $\Z^2$ with $2\times 1$ and $1\times2$ rectangles) there are two
natural conformal coordinates $z\in\H$ and $w\in\bar\H$, related by the ``characteristic polynomial'' $1+z+w-zw=0$, as discussed in \cite{KO1}.
The relationship between $(s,t)$ and $(z,w)$ is given in Figure \ref{dominozw}: it is
$$(s,t) = \frac{2}{\pi}(\arg z-\arg w-\pi, \arg z+\arg w),$$
and $(s,t)\in\N = cvx\{(2,0),(0,2),(-2,0),(0,-2)\}$; see \cite{CKP}.
\begin{figure}
\begin{center}\includegraphics[width=3in]{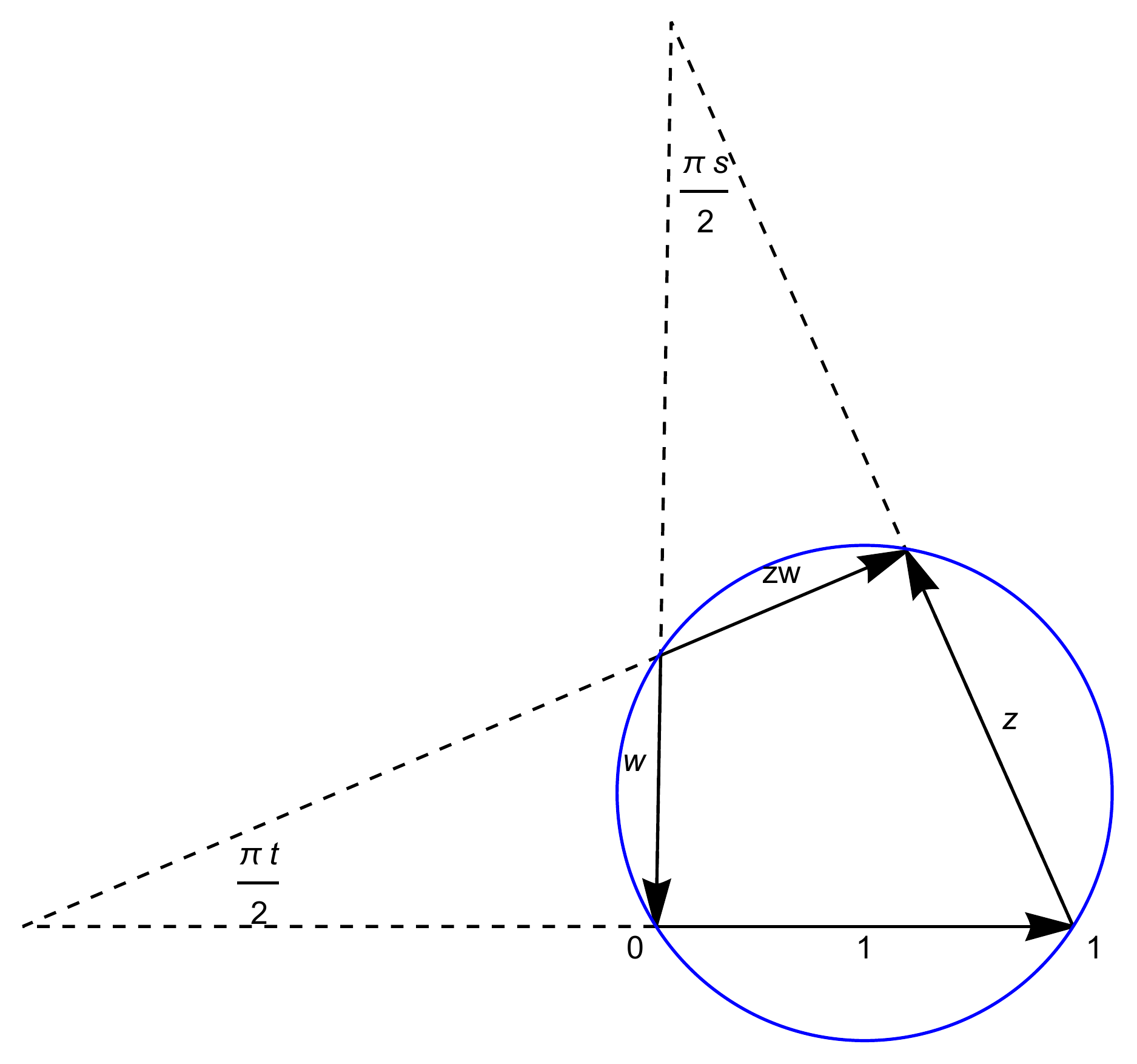}\end{center}
\caption{\label{dominozw} When $1+z+w-zw=0$ (and $z\in\H$), the convex quadrilateral with sides $1,z,-zw,w$ is cyclic.
The values $s,t$ are obtained from the angles of the extended sides as shown. }
\end{figure} 

The Aztec diamond of order $n$ is the region of Figure \ref{AD}, left panel. 
In the limit $n\to\infty$, the height function limit shape has facetted
regions near each corner, where the limit shape is linear with slope at a corner of $\N$. 
The liquid region is a disk tangent
to each of the sides of the limiting polygon. (We do not actually assume the exact shape of the liquid region, this will be found a posteriori.) The situation is as shown in Figure \ref{AD}, right panel: 
the height function along the 
boundary determines the linear equations of the limit shape near each corner. 
In this example the $z$ values run once around $\R\cup\{\infty\}$ as the point $(x,y)$ 
runs around the boundary of the liquid region.
The slopes and intercepts (in blue in the figure) are piecewise constant functions of $z$ as $z$ runs over the real axis, as indicated in the following table. Recall that by Corollary \ref{cor:dimerharmonic} we may assume that $\psi \equiv 1$ in \eqref{eq:G}.
\begin{center}\begin{tabular}{c | c | c | c | c}
&$z<-1$&$-1<z<0$&$0<z<1$&$1<z$\\\hline
$s$&$0$&$2$&$0$&$-2$\\
$t$&$2$&$0$&$-2$&$0$\\
$G$&$1$&$-1$&$1$&$-1$
\end{tabular}\end{center}
\begin{figure}
\begin{center}\includegraphics[width=2.in]{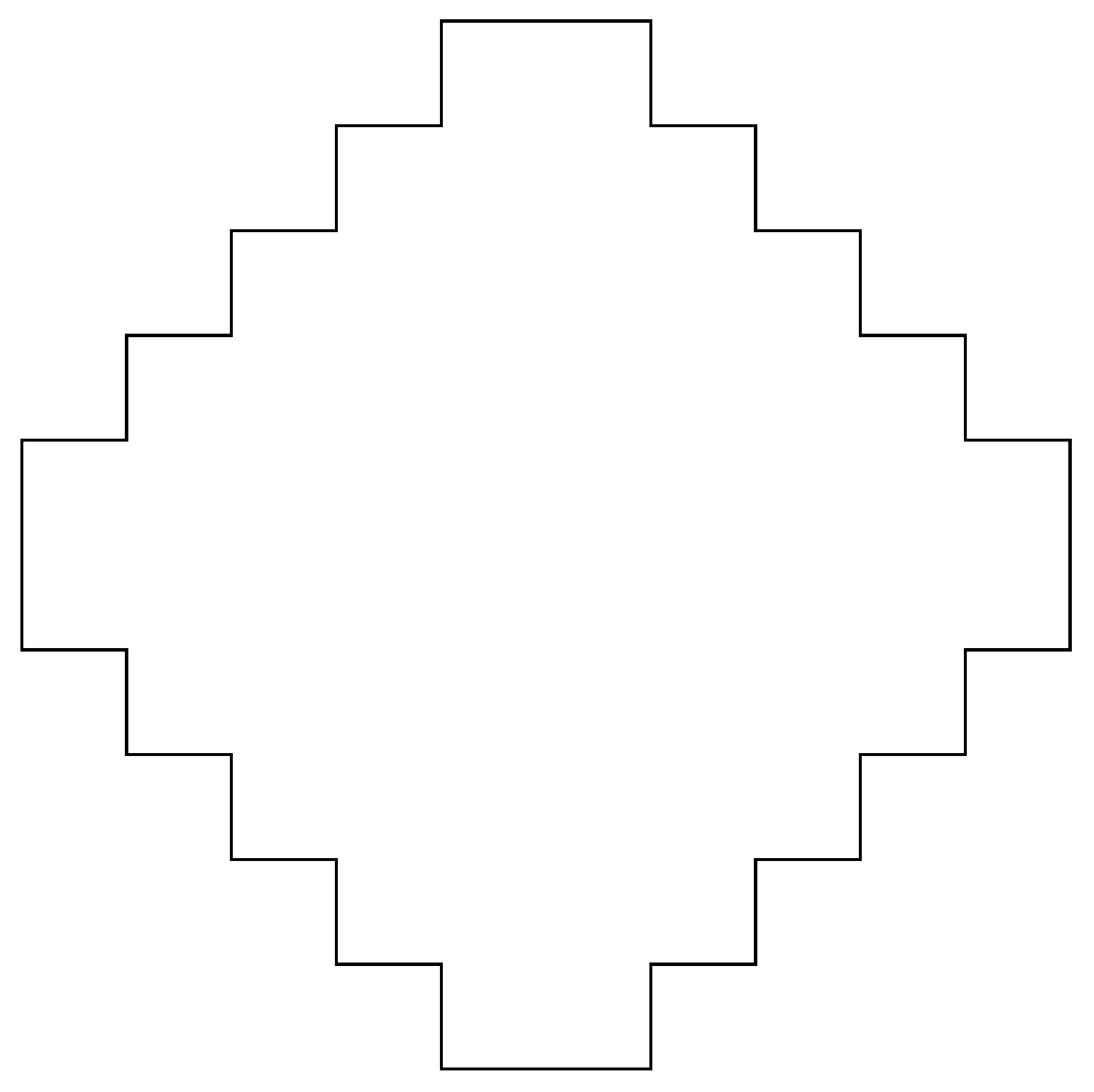}\hskip1cm\includegraphics[width=2.5in]{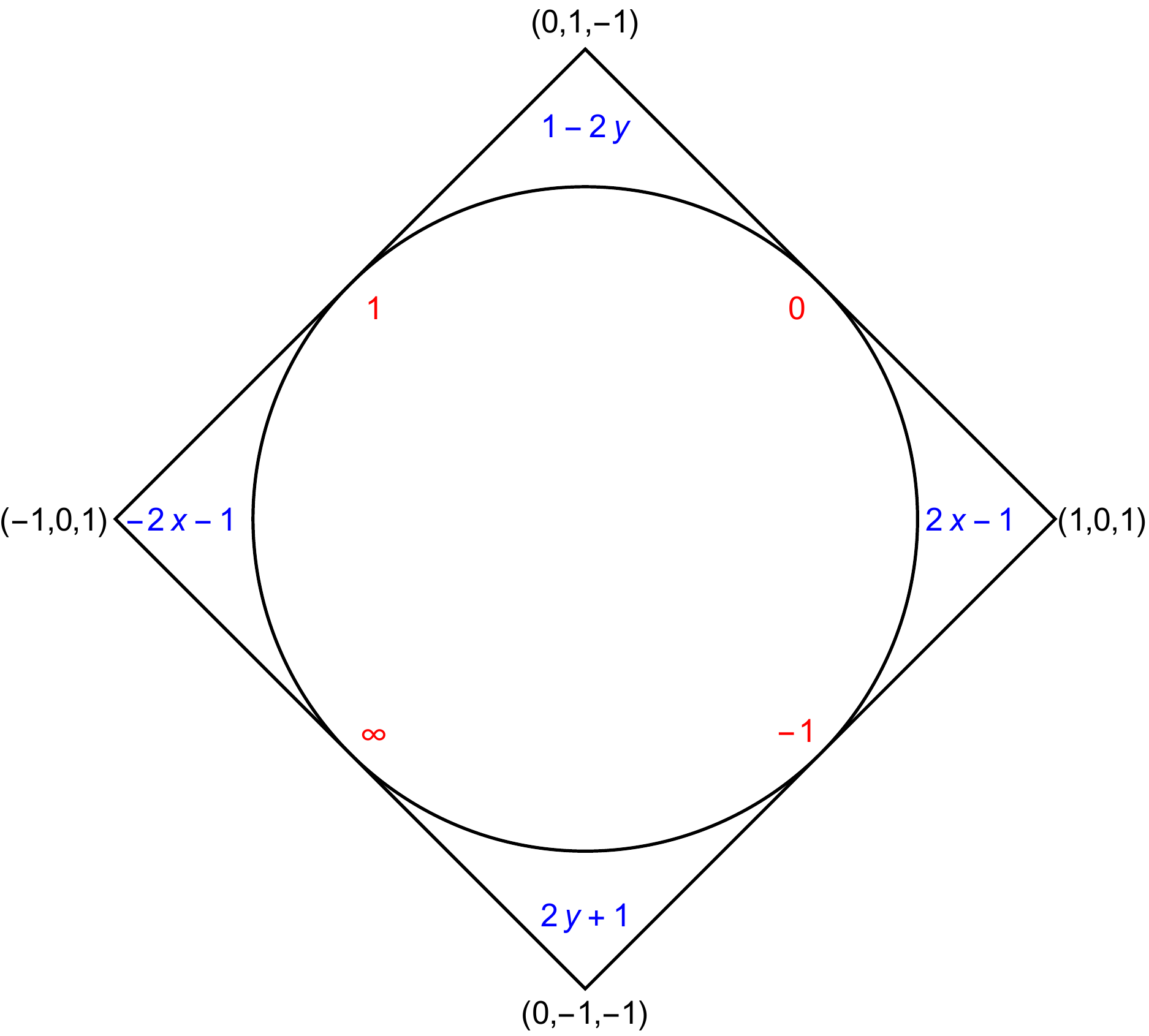}\end{center}
\caption{\label{AD} The Aztec diamond region of order $n=5$, and in the limit $n\to\infty$, 
the limiting heights of the facets (blue) and $z$ values at the tangency points (red). The black values are the positions of
the vertices in $\R^3$; the height (third coordinate) is linear along each boundary segment, and determines the intercepts of the linear functions on the facets.} 
\end{figure} 

Because $G$ is a harmonic function of $z$, we can find it from the harmonic extension of its boundary values on $\H$,
given by the last line in the above table; we have 
$$G = \frac{2}{\pi}(-\frac{\pi}{2}+\arg(z-1)-\arg(z)+\arg(z+1)).$$

The equation for the limit shape (the envelope of the moving planes) is obtained by solving the linear system, see \eqref{eq:complexline}
\begin{align}
sx+ty+G&=x_3\notag\\
s_z x+t_z y+G_z&=0\label{3deq}
\end{align}
for $x,y,x_3$ as functions of $z$. See Figure \ref{AD3d}.
\begin{figure}
\begin{center}\includegraphics[width=3in]{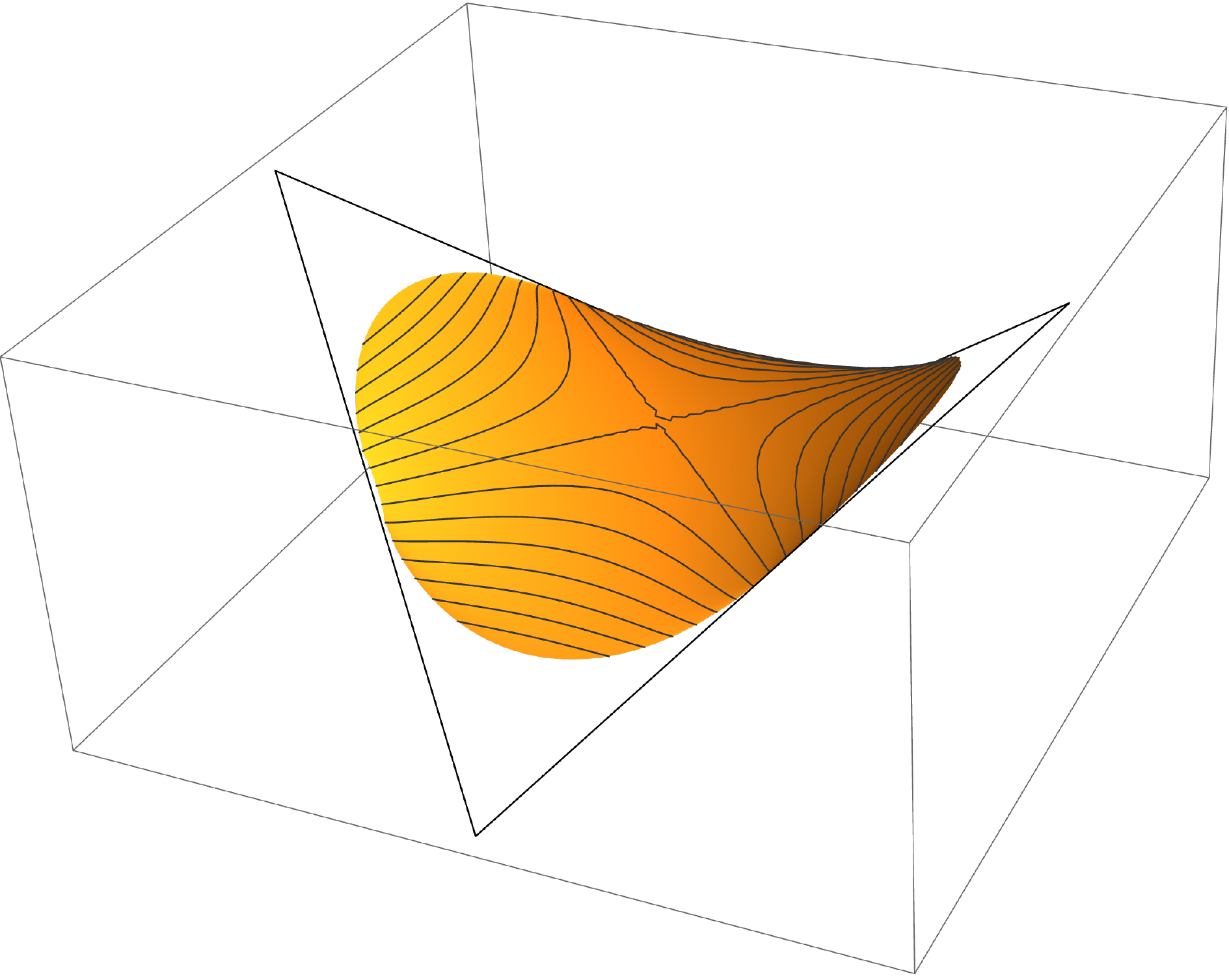}\end{center}
\caption{\label{AD3d} Aztec diamond limit shape (with contour lines for the height function).}
\end{figure} 

By calculating the $\partial_z$-derivatives, \eqref{3deq} takes the explicit form
\[ \left( \frac1z+\frac1{z-1}-\frac1{z+1}\right) x +\left( \frac1z - \frac1{z-1}+\frac1{z+1} \right) y + \frac{1}{z-1}-\frac1z+\frac{1}{z+1}=0. 
\] 
Taking real and imaginary parts, we  find $x$ and $y$ to be
\[x=\frac{1-2\, \Re z-|z|^2}{2(1+|z|^2)}  \quad \text{and} \quad y=\frac{1+2\, \Re z -|z|^2}{2(1+|z|^2)}.
\]
Inverting these relations
\[ z(x,y)=\frac{y-x + i \sqrt{1-2(x^2+y^2)}}{1+x+y}.
\]
The height function can be written down explicitly with the help of the introduced functions as
\[ x_3(x,y)=s(z) x +t(z) y+G(z), \quad x^2+y^2 < \frac12.
\]

We emphasize that in deriving the limit shape the only a priori assumption we made is that it has four facetted regions in the four corners. This assumption is justified rigorously (for this case and in similar settings with more sides) in \cite{KO1} and \cite{ADPZ}. The boundary conditions dictate the equations for these four planes and their harmonic extension in terms of the $z$-variable give the entire family of tangent planes. The exact form of the frozen boundary is then found \emph{a posteriori} from the envelope construction.

\subsection{Domino tilings of $L$-shape}

We consider domino tilings of an $L$-shaped Aztec diamond from \cite{CPS}.
Again the techniques of  \cite{KO1} and \cite{ADPZ} apply to this setting. However, it is highly nontrivial to find either the $Q_0$ function of \cite[Corollary 1]{KO1} or the Blaschke product/holomorphic factor pair $(B,\gamma)$ of \cite[Theorem 5.2]{ADPZ} from the given boundary conditions. Indeed, \cite{CPS} proceeds with the (not fully rigorous) tangent method to find the frozen boundary. We show here that the full limit surface can be obtained  from Corollary \ref{cor:dimerharmonic} with minimal difficulty.

\begin{figure}
\begin{center}\includegraphics[width=4in]{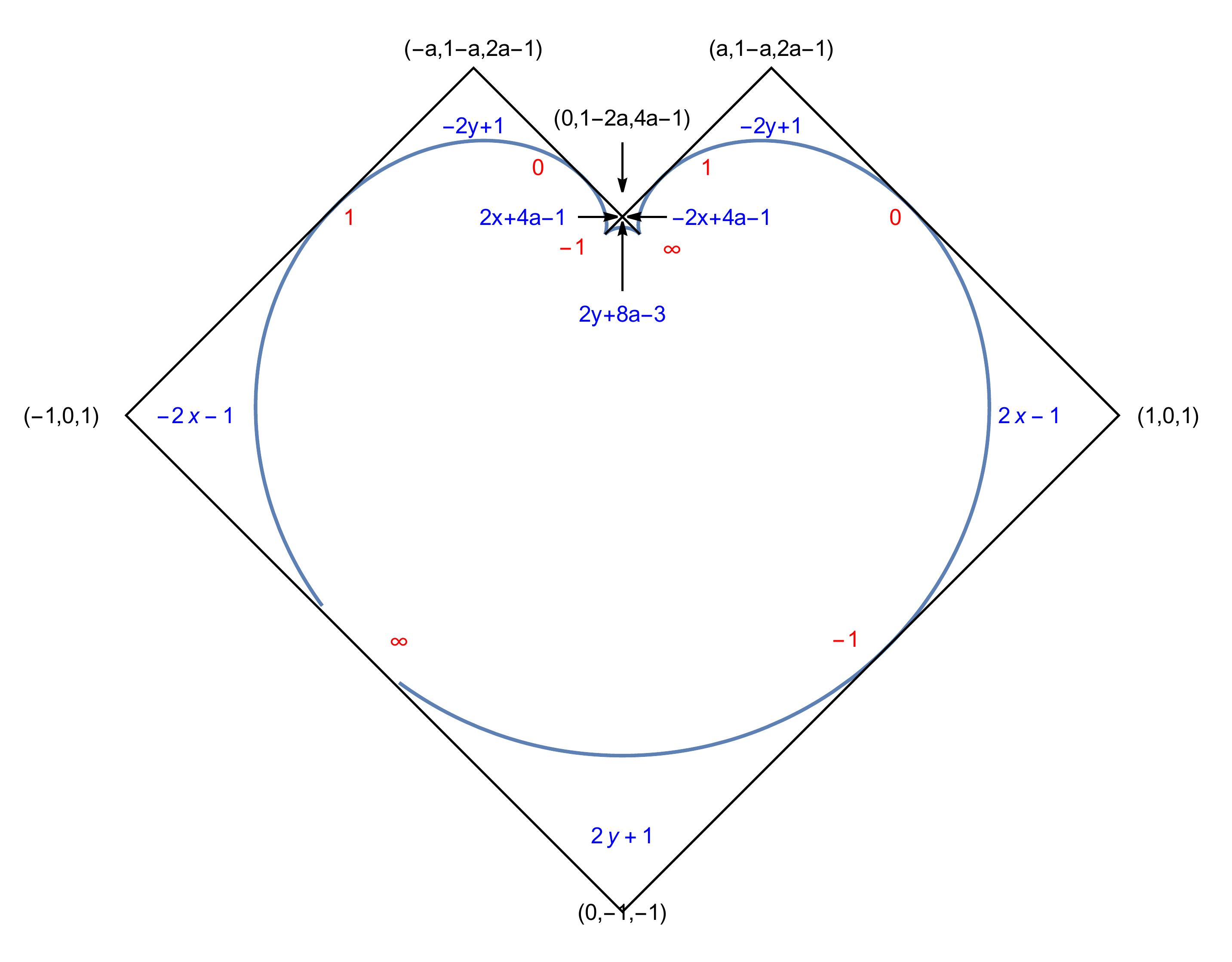}\end{center}
\caption{\label{ADL} Boundary data for domino tilings of an $L$ shaped region.}
\end{figure} 

See Figure \ref{ADL}. The $z$ values are marked in red, and facet equations are in blue; there are $8$ facets.
Let $u\in\H$ parameterize the liquid region; $u$ is a double cover of $z$, that is, $z=z(u)$ 
is a rational function of degree $2$. 
We can normalize
so that $u$ sends $0$ and $\infty$ to $\infty$.
Let $a_1<a_2<a_3<0<a_4<a_5<a_6$ be the points that $u$ sends to $\{-1,0,1,-1,0,1\}$ respectively.
Without loss of generality (after a M\"obius transformation) set $a_3=-1$.
Then $z=z(u)$ has the form $z=\frac{A (u-a_2)(u-a_5)}{u}$ where, solving, the constants 
$A,a_4,a_5,a_6$ are rational functions of $a_1,a_2$. 

\begin{center}\begin{tabular}{c | c | c | c | c|c|c|c|c}
&$P_1$&$P_2$ &$P_3$&$P_4$&$P_5$&$P_6$ &$P_7$&$P_8$\\\hline
$u$&$(-\infty,a_1)$&$(a_1,a_2)$&$(a_2,a_3)$&$(a_3,0)$&$(0,a_4)$&$(a_4,a_5)$&$(a_5,a_6)$&$(a_6,\infty)$\\
$z$&$(-\infty,-1)$&$(-1,0)$&$(0,1)$&$(1,\infty)$&$(-\infty,-1)$&$(-1,0)$&$(0,1)$&$(1,\infty)$\\\hline
$s$&$0$&$2$&$0$&$-2$&$0$&$2$&$0$&$-2$\\
$t$&$2$&$0$&$-2$&$0$&$2$&$0$&$-2$&$0$\\
$G$&$1$&$-1$&$1$&$4a-1$&$8a-3$&$4a-1$&$1$&$-1$
\end{tabular}\end{center}

There is one more nontrivial necessary condition, which is that $G_u$ vanishes at the branch point of 
$u$. This is necessary for the local single-valuedness of the surface. This condition uniquely determines
$(a_1,a_2)$ in the relevant range, and therefore all of $a_1,\dots,a_6$.
For example, when $a=\frac{10-3\sqrt{2}}{16}$, we have
$$(a_1,a_2,a_3,a_4,a_5,a_6) = (-4,-4+\frac{3}{\sqrt{2}},-1,\frac{-1+9\sqrt{2}}7,\frac{8+12\sqrt{2}}7,\frac{-4+36\sqrt{2}}7).$$
At this point it is straightforward to find the harmonic extension of $G$ as in the previous example, and solve the linear system
(\ref{3deq}) to get the equation of the surface; the explicit formulas are a little bulky to include here.
The resulting surface is shown in Figure \ref{ADL3d} for $a=.3$ (the feasible range of $a$ values is $0\le a<1/2$; at $a=1/2$
the surface breaks into two components).

\begin{figure}[h]
\begin{center}\includegraphics[width=3in]{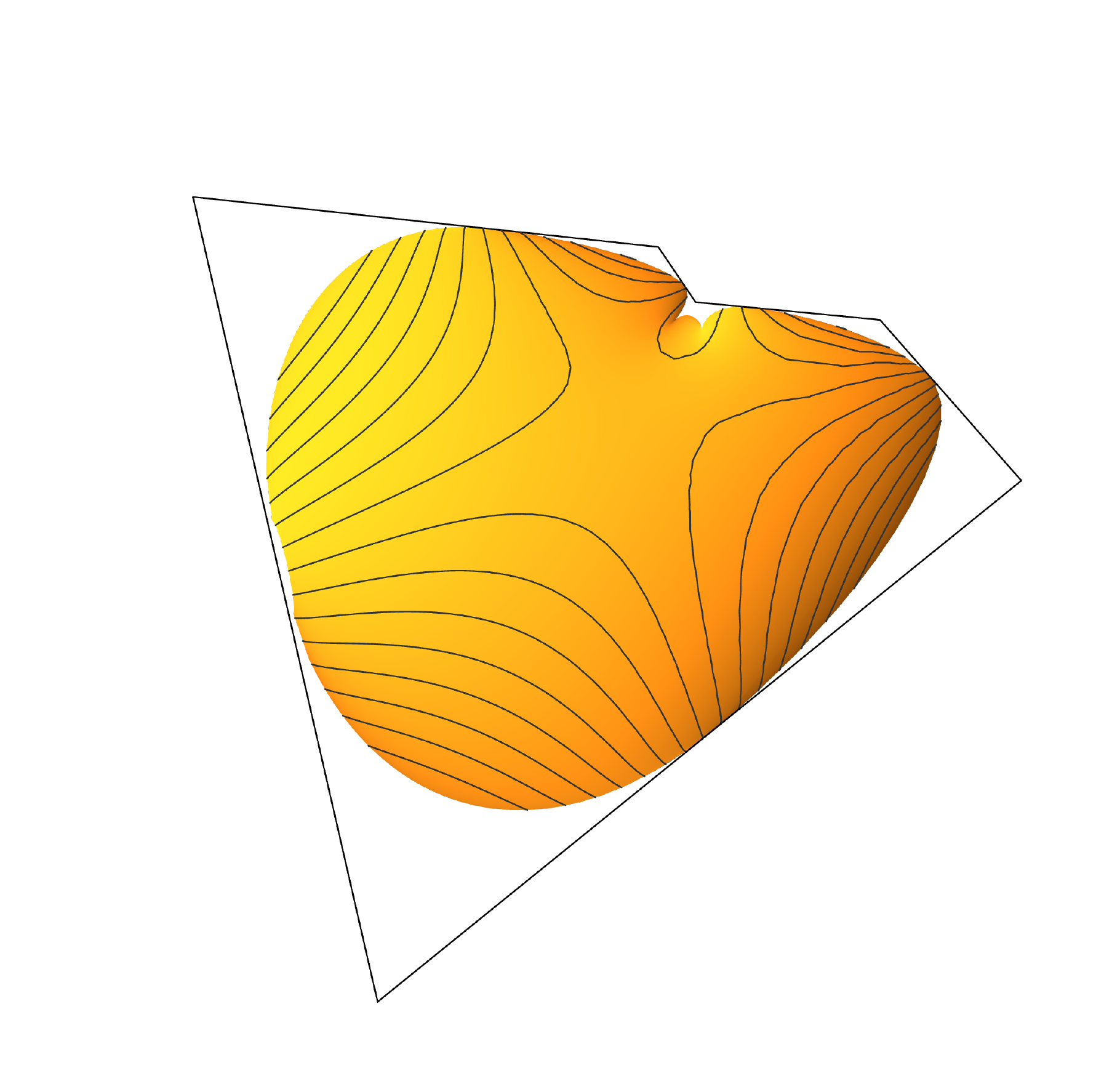}\end{center}
\caption{\label{ADL3d} The limit shape, showing height contours.}
\end{figure}

\bibliographystyle{hplain}
\bibliography{kharmonic}
\end{document}